\documentclass[11pt]{amsart}
\usepackage{amssymb,mathrsfs,graphicx,enumerate}
\usepackage{colortbl}
\definecolor{black}{rgb}{0.0, 0.0, 0.0}
\definecolor{red}{rgb}{1.0, 0.5, 0.5}

\title[From kinetic Cucker-Smale to pressureless Euler]{A rigorous derivation from the kinetic Cucker-Smale model to the pressureless Euler system\\ with nonlocal alignment}

\author[Figalli]{Alessio Figalli}
\address[Alessio Figalli]{\newline ETH Z\"urich, Department of Mathematics, \newline R\"amistrasse 101, 8092 Z\"urich, Switzerland.}
\email{alessio.figalli@math.ethz.ch}

\author[Kang]{Moon-Jin Kang}
\address[Moon-Jin Kang]{\newline Sookmyung Women's University, Department of Mathematics \& Research Institute of Natural Sciences, \newline 100, Cheongpa-ro 47-gil, Yongsan-gu, Seoul, 04310, Korea }
\email{moonjinkang@sookmyung.ac.kr}

\newtheorem{theorem}{Theorem}[section]
\newtheorem{proposition}[theorem]{Proposition}
\newtheorem{lemma}[theorem]{Lemma}

\theoremstyle{definition}
\newtheorem{remark}[theorem]{Remark}

\newcommand{\bbr}{\mathbb R}

\newcommand{\bbt} {\mathbb T}

\numberwithin{equation}{section}

\numberwithin{figure}{section}
\newcommand{\beq}{\begin{equation}}
\newcommand{\eeq}{\end{equation}}
\newcommand{\bsp}{\begin{split}}
\newcommand{\esp}{\end{split}}

\def\eps{\varepsilon }

\newcommand\adots{\mathinner{\mkern2mu\raise1pt\hbox{.}
\mkern3mu\raise4pt\hbox{.}\mkern1mu\raise7pt\hbox{.}}}

\renewcommand{\div}{{\rm div}}

\def\charf {\mbox{{\text 1}\kern-.30em {\text l}}}

\begin{document}

\date{\today}

\subjclass{35Q70, 35B25} \keywords{hydrodynamic limit, kinetic Cucker-Smale model, local alignment, pressureless Euler system, relative entropy, Wasserstein distance}

\thanks{\textbf{Acknowledgment.} 
The work of A. Figalli is supported by the ERC Grant ``Regularity and Stability in Partial Differential Equations (RSPDE)''. The work of M.-J. Kang was supported by the NRF-2017R1C1B5076510 and Sookmyung Women's University Research Grants (1-1703-2045).
}

\begin{abstract}
We consider the kinetic Cucker-Smale model with local alignment as a mesoscopic description for the flocking dynamics. 
The local alignment was first proposed by Karper, Mellet and Trivisa \cite{K-M-T-3}, as a singular limit of a normalized non-symmetric alignment introduced by Motsch and Tadmor \cite{M-T-1}. The existence of weak solutions to this model is obtained in \cite{K-M-T-3}. The time-asymptotic flocking behavior is shown in this article. Our main contribution is to provide a rigorous derivation from mesoscopic to macroscopic description for the Cucker-Smale flocking models. More precisely, we prove the hydrodynamic limit of the kinetic Cucker-Smale model with local alignment towards the pressureless Euler system with nonlocal alignment, under a regime of strong local alignment. Based on the relative entropy method, a main difficulty in our analysis comes from the fact that the entropy of the limit system has no strict convexity in terms of density variable. To overcome this, we combine relative entropy quantities with the 2-Wasserstein distance.
\end{abstract}
\maketitle \centerline{\date}

\tableofcontents

\section{Introduction}
\setcounter{equation}{0}
This article is mainly devoted to providing a rigorous justification on hydrodynamic limit of the kinetic Cucker-Smale model to the pressureless Euler system with nonlocal alignment force. In \cite{C-S}, Cucker and Smale introduced an agent-based model capturing flocking  phenomenon observed within the complex systems such as a flock of birds, school of fish and swarm of insects. The Cucker-Smale (CS) model has received a extensive attention in the mathematical community as well as the physics, biology, engineering and
social science, etc. (See for instance \cite{CCDW,CCR,C-F-R-T,DFT,FHT,HJNXZ,HLSX,H-T,PS,Z-E-P} and the refereces therein.) In \cite{M-T-1}, Motsch and Tadmor proposed a modified Cucker-Smale model by replacing the original CS alignment by a normalized non-symmetric alignment. In \cite{K-M-T-3}, Karper, Mellet, and Trivisa proposed a new kinetic flocking model as a combination of the CS alignment and a local alignment interaction, where the latter was obtained as a singular limit of the non-symmetric alignment introduced by Motsch and Tadmor.

In this article, we consider the kinetic flocking model without Brownian noise, proposed by Karper, Mellet and Trivisa in \cite{K-M-T-1} on $\bbt^d\times\bbr^d$:
\begin{align}
\begin{aligned} \label{main}
&\partial_t f + v\cdot\nabla_x f + \nabla_v\cdot( L[f] f) + \nabla_v\cdot ((u-v)f) =0,\\
&L[f](t,x,v) = \int_{\bbt^d}\int_{\bbr^{d}} \psi(x-y) f(t,y,w) (w-v) \,dw\,dy,\\
&u(t,x)=\frac{\int_{\bbr^d} v f  dv}{\int_{\bbr^d}  f dv},\quad \|f(0)\|_{L^1(\bbt^d\times \bbr^d)}=1.
\end{aligned}
\end{align}
Here $\psi:{\bbt^d}\to \bbr^d$ is a Lipschitz communication weight that is positive and symmetric, i.e., $\psi(x-y)=\psi(y-x)$. The term $\nabla_v\cdot( L[f] f)$ describes a nonlocal alignment due to the original Cucker-Smale flocking mechanism, while the last term $\nabla_v\cdot ((u-v)f)$ describes a local alignment interaction, because of the averaged local velocity $u$. The global existence of weak solutions to \eqref{main} has been proved in \cite{K-M-T-1}. The flocking behaviors of \eqref{main}, however, are not studied so far. We here provide its time-asymptotic behavior. 

As a mesoscopic description, the kinetic model \eqref{main} is posed in $(t,x,v)\in\bbr\times\bbt^{d}\times\bbr^d$, i.e., in $2d+1$ dimensions. This feature provides a accurate description for a significant number of particles. However, its numerical test is very costly with respect to an associated macroscopic description. Hence, it is very important to find a suitable parameter regime on which the complexity of \eqref{main} is reduced. \\

The main goal of this article is to show a singular limit of \eqref{main} in a regime of strong local alignment:
\begin{align}
\begin{aligned} \label{vlasov}
&\partial_t f^{\eps} + v\cdot\nabla_x f^{\eps} + \nabla_v\cdot( L[f^{\eps}] f^{\eps}) +  \frac{1}{\eps}\nabla_v\cdot ((u^{\eps}-v)f^{\eps}) =0,\\
&L[f^\eps](t,x,v) = \int_{\bbt^d}\int_{\bbr^{d}} \psi(x-y) f^{\eps}(t,y,w) (w-v) \,dw\,dy,\\
&u^{\eps}=\frac{\int_{\bbr^d} v f^{\eps}  dv}{\int_{\bbr^d}  f^{\eps} dv},\\
& f^{\eps}|_{t=0} = f^{\eps}_0,\quad \|f_0^\eps\|_{L^1(\bbt^d\times \bbr^d)}=1.
\end{aligned}
\end{align}
As $\eps\to0$, it is expected that the solution $f^{\eps}$ of \eqref{vlasov} converges, in some weak sense, to a mono-kinetic distribution\footnote{In this paper we will use the symbol $\otimes$ in two different contexts: if $\mu$ is a measure on a complete metric space $X$, and $\{\nu_x\}_{x\in X}$ is a family of measures on a complete metric space $Y$, then
$\nu_x\otimes\mu$ denotes the measure on $X\times Y$ defined as
$$
\int_{X\times Y}\varphi\,d[\nu_x\otimes \mu]=
\int_{X}\biggl(\int_Y\varphi(x,y)\,d\nu_x(y)\biggr)\,d\mu(x)\qquad \forall\,\varphi\in C_c(X\times Y).
$$
When $\nu_x$ is independent of $x$ (that is, $\nu_x=\nu$ for all $x$), we use the more standard notation $
\mu\otimes \nu$ (instead of $\nu\otimes \mu$, as done before) to denote the product measure:
$$
\int_{X\times Y}\varphi\,d[\mu\otimes \nu]=
\int_{X}\biggl(\int_Y\varphi(x,y)\,d\nu(y)\biggr)\,d\mu(x)\qquad \forall\,\varphi\in C_c(X\times Y).
$$
Finally, if $a,b\in \bbr^d$ are vectors, then $a\otimes b$ denotes the $d\times d$-matrix with entries
$$
(a\otimes b)_{ij}=a_ib_j\qquad \forall\, i,j=1,\ldots,d.
$$
The meaning will always be clear from the context.}
\beq\label{mono-ki}
\delta_{v=u(t,x)}\otimes\rho(t,x).
\eeq
Here, $\delta_{v=u(t,x)}$ denotes a Dirac mass in $v$ centered on $u(t,x)$.
Also, as we shall explain later, at least formally $\rho$ and $u$ should solve the associated limit system given by the pressureless Euler system with nonlocal flocking dissipation:
\begin{align}
\begin{aligned} \label{PE}
&\partial_t \rho + \nabla\cdot (\rho u) = 0,  \\
&\partial_t (\rho u) + \nabla \cdot (\rho u\otimes u) = \int_{\bbt^d} \psi(x-y)\rho(t,x)\rho(t,y) (u(t,y) - u(t,x)) \, dy,\\
&\rho|_{t=0} = \rho_0,\quad u |_{t=0} = u_0,\quad \|\rho_0\|_{L^1(\bbt^d)}=1.
\end{aligned}
\end{align}
The main difficulty  in the justification of this limit comes from the singularity of the mono-kinetic distribution. To the best of our knowledge, there is no general method to handle the hydrodynamic limit from some kinetic equations to the pressureless Euler systems, no matter what regime is considered. Indeed, there are few results on this kinds of limit, see \cite{JR,K18,KV1} (see also \cite{Jab} for a general treatment of similar regimes that lead to Dirac formation and pressureless gases equations). 

It is worth mentioning that the pressureless Euler system without the nonlocal alignment has been used for the formation of large-scale structures in astrophysics and the aggregation of sticky particles \cite{SSZ,Ze}.  For more theoretical studies on the pressureless gases, we for example refer to \cite{Bouchut,BJ,Boudin,BG,HW,PR,WRS}.

The macroscopic flocking model \eqref{PE} or its variants have been formally derived under a mono-kinetic ansatz \eqref{mono-ki}, and studied in various topics (see for example \cite{DKRT,H-K-K1,H-K-K2,HHW,TT}).   In \cite{H-K-K1}, the authors have shown the global well-posedness of \eqref{PE} with a suitably smooth and small initial data, and the time-asymptotic flocking behavior. In \cite{H-K-K2}, the authors dealt with a moving boundary problem of \eqref{PE} with compactly supported initial density. We also refer to \cite{HHW} on a reformulation of \eqref{PE} into hyperbolic conservation laws with damping in one dimension.

In \cite{K-M-T-2}, the author have shown the hydrodynamic limit of the kinetic flocking model \eqref{main} with Brownian motion, that is, Vlasov-Fokker-Planck type equation, under the regime of strong local alignment and strong Brownian motion:
\begin{align}
\begin{aligned} \label{model-1}
\partial_t f^{\eps} + v\cdot\nabla_x f^{\eps} + \nabla_v\cdot(L[f^\eps] f^{\eps} ) +  \frac{1}{\eps}\nabla_v\cdot ((u^{\eps}-v)f^{\eps}) -\frac{1}{\eps}\Delta_v f^{\eps} =0.
\end{aligned}
\end{align}
In this case, as $\eps\to 0$, $f^{\eps}$ converges to a smooth local equilibrium given by a local Maxwellian, contrary to \eqref{mono-ki}. There, the authors used the relative entropy method, heavily relying on a strict convexity of the entropy of the isothermal Euler system (as a limit system of \eqref{model-1}):
\begin{align*}
\begin{aligned}
&\partial_t \rho + \nabla\cdot (\rho u) = 0,  \\
&\partial_t (\rho u) + \nabla \cdot (\rho u\otimes u)+\nabla \rho = \int_{\bbt^d} \psi(x-y)\rho(t,x)\rho(t,y) (u(t,y) - u(t,x)) \, dy.
\end{aligned}
\end{align*}
The relative entropy method based on a strict convex entropy has been successfully used to prove the hydrodynamic limit of Vlasov-Fokker-Planck type equations, we refer to \cite{BV,C-C-K,GJV,MV,Va}.

On the other hand, the pressureless Euler system \eqref{PE} has a convex entropy given by 
\beq\label{entropy-11}
\eta(\rho, \rho u)=\rho\frac{|u|^2}{2},
\eeq
which is not strictly convex with respect to $\rho$. For this reason, the associated relative entropy \eqref{entropy-11} is not enough to control the convergence of the nonlocal alignment term
(compare with \cite{KV1}, where the nonlocal alignment is not present).
To overcome this difficulty, we first estimate a $L^2$-distance of characteristics generated by vector fields $u^{\eps}$ and $u$, that controls the 2-Wasserstein distance of densities, and then combine the estimates of the relative entropy and the $L^2$-distance of characteristics.

As a related work on \eqref{model-1}, we refer to \cite{C-C-K}, where the authors studied the flocking behavior and hydrodynamic limit of a  coupled system of \eqref{model-1} and fluid equations via drag force.\\

The rest of this paper is organized as follows. In Section 2, we mention different scales of Cucker-Smale models from microscopic level to macroscopic level, and then specify some known existence results on the two descriptions \eqref{main} and \eqref{PE}. In Section 3, we present our main theorem on the hydrodynamic limit, and collect some useful results on the relative entropy method and the optimal transportation theory that are used in the proof of the main theorem. In Section 4, we present some structural hypotheses to guarantee hydrodynamic limit in a general setting. Then we apply the general result to our systems by verifying the hypotheses in Section \ref{sect:proof thm}. In the Appendix, we provide the proof of the long time-asymptotic flocking dynamics and the existence of mono-kinetic solutions for the kinetic model \eqref{main}.

\section{Various scales of Cucker-Smale models}
\setcounter{equation}{0}
In this section we first present various scales of Cucker-Smale models, from microscopic level to macroscopic level. Then we state some known results on global existence of weak solutions to the kinetic description \eqref{main}, and local existence of smooth solutions to the limit system \eqref{PE}. Those results are crucially used in the proof of the main theorem. Finally, in Theorem \ref{prop-flocking}, we present the time-asymptotic flocking behavior of the kinetic model \eqref{main}.

\subsection{Variants of Cucker-Smale models}
In this subsection, we briefly present the kinetic CS model and its variants.  Cucker and Smale in \cite{C-S} proposed a mathematical model to explain the flocking phenomenon: 
\begin{align} \label{CS}
\begin{aligned}
\frac{dx_i}{dt} &= v_i, \quad i =1, \cdots, N, \\
\frac{dv_i}{dt} &= \frac{1}{N} \sum_{j=1}^{N} \psi(x_j - x_i) (v_j - v_i),
\end{aligned}
\end{align}
where $x_i, v_i \in \bbr^d$ denote the spatial position and velocity of the $i$-th particle for an ensemble of $N$ self-propelled particles. The kernel $\psi(|x_j - x_i|)$ is a communication weight given by
\begin{equation} \label{comm}
\psi(x_j - x_i) = \frac{\lambda}{(1 + |x_j - x_i|^2)^{\beta}}, \quad \beta \geq 0,~ \lambda>0.
\end{equation}
The system \eqref{CS} with \eqref{comm} was used as an analytical description of the Vicsek model \cite{Vicsek} without resorting to the first principle of physics.

When the number of particles is sufficiently large, the ensemble of particles can be described by the one-particle density function $f = f(t,x,v)$ at the spatial-velocity position $(x, v) \in \bbr^d \times \bbr^d$ at time $t$. Then, the evolution of $f$ is governed by the following  Vlasov type equation:
\begin{align}
\begin{aligned} \label{KCS}
& \partial_t f + v \cdot \nabla_x f + \nabla_{v} \cdot (L[f] f) = 0,\\
&L[f](t,x,v) = \int_{\bbr^{2d}} \psi(x-y) f(t,y,w) (w-v) \,dw\,dy.
\end{aligned}
\end{align}
This was first introduced by Ha and Tadmor \cite{H-T} using the BBGKY Hierarchy from the particle CS model \eqref{CS}. A rigorous mean-field limit was given in \cite{H-L}.\\

In \cite{M-T-1}, Motsch and Tadmor recognized a drawback of the CS model \eqref{CS}, which is due to the normalization factor $\frac{1}{N}$. More precisely, when a small group of agents are located far away from a much larger group of agents, the internal dynamics of the small group is almost halted since the total number of agents is relatively very large. To solve this issue, they replaced the nonlocal alignment $L[f]$ by a normalized non-symmetric alignment operator:
\[
\overline L[f](t,x,v) := \frac{\int_{\bbr^{2d}} K^{r}(x-y) f(t,y,w)(w-v)\,dw\,dy}{\int_{\bbr^{2d}} K^{r}(x-y) f(t,y,w)\,dw\,dy},
\]
where the kernel $K^r$ is a communication weight and $r$ denotes the radius of influence of $K^r$. \\
In \cite{K-M-T-3}, the authors considered the case when the communication weight is extremely concentrated nearby each agent, so that the alignment term $\overline L[f]$ corresponds to a short-range interaction. More precisely, they rigorously justified the singular limit $r\rightarrow 0$, i.e., as $K^r$ converges to the Dirac distribution $\delta_0$, in which case $\overline L[f]$
converge to a local alignment term:
\[
\overline L[f](t,x,v) \rightarrow \frac{\int_{\bbr^{d}} f(t,x,w)(w-v)\,dw}{\int_{\bbr^{d}} f(t,x,w)\,dw}= u(t,x)-v,
\] 
where $u(t,x)$ denotes the averaged local velocity defined as $u(t,x)=\frac{\int_{\bbr^d} v f(t,x,v)  dv}{\int_{\bbr^d} f(t,x,v)dv}$. Hence, their new model became \eqref{main}, which consists of two kinds of alignment force: a nonlocal alignment due to the original CS model, plus a local alignment.\\

\subsection{Existence of weak solutions to \eqref{vlasov}} In \cite{K-M-T-1}, the authors showed the existence of weak solutions to the kinetic Cucker-Smale model with local alignment, noise, self-propulsion, and friction:
\begin{align}
\begin{aligned} \label{KMT}
&\partial_t f + v\cdot\nabla_x f +\nabla_v\cdot (L[f] f)+ \nabla_v\cdot ((u-v)f)\\
&\qquad = \sigma\Delta_v f -  \nabla_v\cdot ((a-b|v|^2)vf) ,\\
&L[f] = \int_{\bbr^{2d}} \psi(x-y) f(t,y,w) (w-v) \,dw\,dy,
\end{aligned}
\end{align}
where the kernel $\psi$ is the same as \eqref{vlasov} and $a$, $b$, and $ \sigma$ are nonnegative constants. By their result applied with  $a=b= \sigma=0$ inside the periodic domain $\bbt^d$, 
we obtain existence of solutions for \eqref{vlasov}. To precisely state such existence result, we need to define a (mathematical) entropy $\mathcal{F}(f^{\eps})$ and kinetic dissipations $ \mathcal{D}_1 (f^{\eps})$, $\mathcal{D}_2(f^{\eps})$ for \eqref{vlasov}:
\begin{align}
\begin{aligned}\label{FDD} 
&\mathcal{F}(f^{\eps}):=\int_{\bbr^d} \frac{|v|^2}{2} f^{\eps} dv,\\
& \mathcal{D}_1(f^{\eps}) : = \int_{\bbt^d\times \bbr^d} f^{\eps} |u^{\eps}-v|^2 \,dv\,dx,\\
&  \mathcal{D}_2(f^{\eps}) : = \frac{1}{2}\int_{\bbt^{2d}\times \bbr^{2d}} \psi(x-y) f^{\eps}(x,v) f^{\eps}(y,w) |v-w|^2 \,dx\,dy\,dv\,dw.
\end{aligned}
\end{align}
\begin{proposition}\label{prop-weak} 
For any $\eps>0$, assume that $f^{\eps}_0$ satisfies 
\beq\label{weak-initial}
f^{\eps}_0\ge0,\quad f^{\eps}_0 \in L^{1}\cap L^{\infty}(\bbr^{2d}),\quad |v|^2 f^{\eps}_0 \in L^{1}(\bbr^{2d}).
\eeq
Then there exists a weak solution $f^{\eps}\ge 0$ of \eqref{vlasov} such that
\begin{align}
\begin{aligned} \label{weak-reg}
& f^{\eps}\in C(0,T; L^1(\bbr^{2d}))\cap L^{\infty}((0,T)\times\bbr^{2d}),\\
& |v|^2 f^\eps \in L^{\infty}(0,T; L^1(\bbr^{2d})),
\end{aligned}
\end{align}
and \eqref{vlasov} holds in the sense of distribution, that is, for any $\varphi\in C_c^{\infty} ([0,T)\times\bbr^{2d})$, the weak formulation holds:
\begin{align}
\begin{aligned} \label{weak-form}
&   \int_0^t\int_{\bbr^{2d}} f^{\eps} \Big( \partial_t\varphi + v\cdot\nabla_x\varphi +L[f^{\eps}]\cdot\nabla_v\varphi+\frac{1}{\eps}(u^{\eps}-v) \cdot\nabla_v\varphi \Big)\,dv\,dx\,ds\\
&\qquad +\int_{\bbr^{2d}} f^{\eps}_0\varphi(0,\cdot) \,dv\,dx =0.
\end{aligned}
\end{align}
Moreover, $f^{\eps}$ preserves the total mass and satisfies  the entropy inequality
\begin{align}
\begin{aligned}\label{kinetic-ineq}
&\int_{\bbt^d} \mathcal{F}(f^{\eps})(t)\,dx + \frac{1}{\eps} \int_0^t  \mathcal{D}_1 (f^{\eps})(s)\,ds +  \int_0^t  \mathcal{D}_2 (f^{\eps})(s)\,ds \le \int_{\bbt^d} \mathcal{F}(f^{\eps}_0)\,dx.
\end{aligned}
\end{align}
\end{proposition}
The entropy inequality \eqref{kinetic-ineq} is crucially used in the proof of Theorem \ref{main thm}.

\subsection{Flocking behavior of the kinetic model \eqref{main}} We now present the time-asymptotic flocking behavior of solutions to the kinetic model \eqref{main}. For that, we define the following two Lyapunov functionals: 
\begin{align*}
\begin{aligned} 
&{\mathcal E}_1(t) := \int_{\bbt^d\times\bbr^d}f(t,x,v) |u(t,x)-v|^2 \,dv\,dx,\\
&{\mathcal E}_2(t) :=\int_{\bbt^{2d}}\rho(t,x)\rho(t,y) |u(t,x) - u(t,y)|^2 \,dx\,dy,
\end{aligned}
\end{align*}
where $\rho(t,x)=\int_{\bbr^d}f(t,x,v) dv$.
We remark that ${\mathcal E}_1$ measures a local alignment, and ${\mathcal E}_2$ measures alignment of the averaged local velocities. Then, for the flocking estimate, we combine the two functionals as follows:
\begin{equation} \label{Lya}
{\mathcal E}(t) := {\mathcal E}_1(t) + \frac{1}{2}{\mathcal E}_2 (t).
\end{equation} 
\begin{theorem} \label{prop-flocking}
Let $f$ be a solution to \eqref{main}. Then, we have the time-asymptotic flocking estimate
\begin{equation}
\label{eq:E}
{\mathcal E}(t) \le {\mathcal E}(0)\exp{\Big(-2\min\{1, \psi_m \} t\Big)},\quad t>0,
\end{equation}
where $\psi_m$ is the minimum communication weight:
\[ \psi_m := \min_{x, y \in \bbt^d } \psi(x- y) > 0. \]
In addition, if $u$ is uniformly Lipschitz continuous on a time interval $[0,T]$, namely $\ell_T:=\sup_{t\in[0,T]}\|\nabla_x u\|_{L^\infty(\mathbb T^d)}<\infty$, then
\begin{equation}
\label{eq:E1}
\mathcal E_1(t)\leq \mathcal E_1(0)e^{2(\ell_T-1)},\quad \forall\,t\in [0,T].
\end{equation}
\end{theorem}
\begin{proof}
We postpone the proof to the Appendix.
\end{proof}

\begin{remark}
As an interesting consequence of \eqref{eq:E1}
one obtains that, for smooth solutions, $\mathcal E_1(0)=0$ implies that $\mathcal E_1(t)=0$ for all $t \in [0,T]$. In other words, mono-kinetic initial conditions remain mono-kinetic as long as the velocity field is Lipschitz. One can note that mono-kinetic solutions to \eqref{main} simply correspond to solutions of the pressureless Euler system \eqref{PE}, hence the short time existence of Lipschitz solutions is guaranteed by Proposition \ref{prop-strong} and Remark \ref{rmk:lip} below.
\end{remark}
\subsection{Formal derivation of the hydrodynamic Cucker-Smale system \eqref{PE}}
We consider the hydrodynamic variables $\rho^{\eps}:=\int_{\bbr^d}f^{\eps} dv$ and $\rho^{\eps}u^{\eps}:=\int_{\bbr^d}vf^{\eps} dv$. \\
First of all, integrating \eqref{vlasov} with respect to  $v$, we get the continuity equation:
\[
\partial_t\rho^{\eps} + \nabla_x \cdot (\rho^{\eps} u^{\eps})=0.
\] 
Multiplying \eqref{vlasov} by $v$, and then integrating it with respect to  $v$, we have
\[
\partial_t(\rho^{\eps}u^{\eps}) + \nabla_x \cdot \Big(\int_{\bbr^d} v\otimes vf^{\eps} dv\Big)  =\int_{\bbt^d} \psi(x-y)\rho^{\eps}(t,x)\rho^{\eps}(t,y) (u^{\eps}(t,y) - u^{\eps}(t,x))\,dy,
\] 
where we used $u^{\eps}=\frac{\int_{\bbr^d} v f^{\eps}  dv}{\int_{\bbr^d}  f^{\eps} dv}$.\\
Then, we rewrite the system for $\rho^{\eps}$ and $u^{\eps}$ as
\begin{align}
\begin{aligned} \label{nonclosed}
&\partial_t\rho^{\eps} + \nabla_x \cdot (\rho^{\eps} u^{\eps})=0,\\
&\partial_t(\rho^{\eps}u^{\eps}) + \nabla_x \cdot (\rho^{\eps} u^{\eps} \otimes  u^{\eps} +P^{\eps})  =\int_{\bbt^d} \psi(x-y)\rho^{\eps}(t,x)\rho^{\eps}(t,y) (u^{\eps}(t,y) - u^{\eps}(t,x))\,dy,\\
\end{aligned}
\end{align}
where $P^{\eps}$ is the stress tensor given by
\[
P^{\eps}:= \int_{\bbr^d} (v-u^{\eps})\otimes (v-u^{\eps}) f^{\eps} dv.
\]
If we take $\eps\to 0$ in \eqref{vlasov}, the local alignment term $\nabla_v\cdot ((u^{\eps}-v)f^{\eps})$ converges to $0$.
Hence, if $\rho^{\eps}\to\rho$ and $\rho^{\eps}u^{\eps}\to \rho u$ for some limiting functions $\rho$ and $u$, we have that $f^{\eps}\to \delta_{v=u}\otimes\rho$ (in some suitable sense).
Hence, the stress tensor $P^{\eps}$ should vanish in the limit, since
\[
\int_{\bbr^d} (v-u)\otimes (v-u) \delta_{v=u}\rho\,dv=0.
\]
Therefore, at least formally, the limit quantities $\rho$ and $u$ satisfy the pressureless Euler system with nonlocal alignment:
\begin{align*}
\begin{aligned} 
&\partial_t\rho + \nabla_x \cdot (\rho u)=0,\\
&\partial_t(\rho u) + \nabla_x \cdot (\rho u\otimes u)  =\int_{\bbt^d} \psi(x-y)\rho(t,x)\rho(t,y) (u(t,y) - u(t,x))\,dy.\\
\end{aligned}
\end{align*}

\subsection{Existence of classical solutions to \eqref{PE}} 
We present here the local existence of classical solutions to the pressureless Euler system \eqref{PE}.  

\begin{proposition}\label{prop-strong} 
Assume that 
\beq\label{strong-initial}
\rho_0> 0\quad \mbox{in}~ \bbt^d\quad\mbox{and}\quad (\rho_0, u_0)\in H^{s}(\bbt^d)\times H^{s+1}(\bbt^d)\quad\mbox{for}~s>\frac{d}{2}+1.
\eeq
Then, there exists $T_*>0$ such that \eqref{PE} has a unique classical solution $(\rho, u)$ satisfying
\begin{align}
\begin{aligned}\label{sol-limit}
& \rho\in C^0([0,T_*];H^{s}(\bbt^d))\cap C^1([0,T_*];H^{s-1}(\bbt^d)),\\
&u\in C^0([0,T_*];H^{s+1}(\bbt^d))\cap C^1([0,T_*];H^{s}(\bbt^d)).
\end{aligned}
\end{align}
\end{proposition}
\begin{remark}
\label{rmk:lip}
Since $s>\frac{d}{2}+1$, by Sobolev inequality it follows that $(\rho, u)\in C^1([0,T_*]\times \bbt^d)$.
\end{remark}
Proposition \ref{prop-strong} has been proven in \cite{H-K-K1}. There, the authors obtained also a global well-posedness of classical solutions, provided an initial datum is suitably smooth and small.

\section{Main result and Preliminaries}
In this section, we first present our main result on the hydrodynamic limit of \eqref{vlasov}. We next present useful results on the relative entropy method and the optimal transportation theory, which are used as main tools in the next section.
\subsection{Main result} 
For the hydrodynamic limit, we consider a well-prepared initial data $f^{\eps}_0$ satisfying \eqref{weak-initial} and
\begin{itemize}
\item
(${\mathcal A}1$): $\int_{\bbt^d} \int_{\bbr^d} \Big(f^{\eps}_0\frac{|v|^2}{2} - \rho_0\frac{|u_0|^2}{2} \Big)
\,dv\, dx =\mathcal{O}(\eps),$
\vspace{0.2cm}
\item
(${\mathcal A}2$): $\|\rho^{\eps}_0 - \rho_0\|_{L^1(\bbt^d)}=\mathcal O(\eps)$,
\vspace{0.2cm}
\item
(${\mathcal A}3$): $ \|u^{\eps}_0 - u_0\|_{L^{\infty}(\bbt^d)} =\mathcal{O}({\eps})$.
\end{itemize}

We now specify our main result on the hydrodynamic limit.
\begin{theorem}\label{main thm} 
Assume that the initial data $f^{\eps}_0$ and $(\rho_0, u_0)$ satisfy \eqref{weak-initial}, \eqref{strong-initial}, and $({\mathcal A}1)$-$({\mathcal A}3)$.
Let $f^{\eps}$ be a weak solution to \eqref{vlasov} satisfying \eqref{kinetic-ineq}, and $(\rho, u)$ be a local-in-time smooth solution to \eqref{PE} satisfying \eqref{sol-limit} up to the time $T_*$. Then, there exists a positive constant $C_*$ (depending on $T_*$) such that, for all $t\le T_*$,
\begin{align}
\begin{aligned}\label{main-ineq}
\int_{\bbt^d}\rho^{\eps} (t) |(u^{\eps} - u)|^2(t)\, dx +W^2_2 (\rho^{\eps}(t),\rho(t)) \le C_*\eps,
\end{aligned}
\end{align}
where $\rho^{\eps}=\int_{\bbr^d} f^{\eps} dv$, $\rho^{\eps}u^{\eps} = \int_{\bbr^d} v f^{\eps} dv,$ and $W_2$ denotes the 2-Wasserstein distance.\\
Therefore, we have
\beq\label{w-conv}
f^{\eps} \rightharpoonup  \delta_{v=u(t,x)}\otimes \rho(t,x)\quad \quad \mbox{in} ~\mathcal{M}((0,T_*)\times\bbt^d\times\bbr^d),
\eeq
where $\mathcal{M}((0,T_*)\times\bbt^d\times\bbr^d)$ is the space of nonnegative Radon measures on $(0,T_*)\times\bbt^d\times\bbr^d$.
\end{theorem}

The proof of this result is postponed to Section \ref{sect:proof thm}. In the next subsections we collect some preliminary facts that will be used later in the proof.
\subsection{Relative entropy method} 
First of all, we rewrite the limit system \eqref{PE} in an abstract form using the notation
\begin{align*}
\begin{aligned}
&P=\rho u,\quad U={\rho \choose P},\quad A(U)={ P^T \choose \frac{P\otimes P}{\rho}},\\
 & F(U)={0 \choose  \int_{\bbt^d} \psi(x-y)\rho(t,x)\rho(t,y) (u(t,y) - u(t,x))\,dy}.
\end{aligned}
\end{align*}
Then we can rewrite \eqref{PE} as the balance law
\beq\label{system}
\partial_t U + {\rm div}_x A(U) = F(U).
\eeq
We consider the relative entropy and relative flux:
\begin{align}
\begin{aligned}\label{relative}
&\eta(V|U)=\eta(V)-\eta(U)- D\eta(U)\cdot(V-U),\\
&{A}(V|U)={A}(V)-{A}(U)- D{A}(U)\cdot(V-U),
\end{aligned}
\end{align}
where $D{A}(U)\cdot(V-U)$ is a matrix defined as
\[
(D{A}(U)\cdot(V-U))_{ij} = \sum_{k=1}^{d+1} \partial_{U_k}{A}_{ij}(U)(V_k-U_k),\quad 1\le i\le d+1,\quad 1\le j\le d.
\]
By the theory of conservation laws, the system \eqref{system} has a convex entropy $\eta(U)=\rho\frac{|u|^2}{2}$ with entropy flux $G$ given by the identity:
\[
\partial_{U_i}  G_{j} (U) = \sum_{k=1}^{d+1}\partial_{U_k} \eta(U) \partial_{U_i}  A_{kj} (U),\quad 1\le i\le d+1,\quad 1\le j\le d.
\]
Since $\eta(U)=\frac{|P|^2}{2\rho}$, and
\beq\label{D-eta}
D\eta(U) ={D_{\rho}\eta \choose D_P \eta}= { -\frac{|P|^2}{2\rho^2} \choose  \frac{P}{\rho}} = { -\frac{|u|^2}{2} \choose  u},
\eeq
for given $V={q \choose qw}$, $U={\rho \choose \rho u}$, we have
\begin{align}
\begin{aligned}\label{relative-e}
\eta(V|U)&=\frac{q}{2}|w|^2 -\frac{\rho}{2}|u|^2 +\frac{|u|^2}{2} (q-\rho) - u (qw-\rho u)\\
&=\frac{q}{2}|u-w|^2.
\end{aligned}
\end{align}

The next proposition provides a cornerstone to verify the hydrodynamic limit through the relative entropy method. For its proof, we refer to the proof of Proposition 4.2 in \cite{K-M-T-2} (See also \cite{Va}).
\begin{proposition}\label{prop-key2} 
Let $U$ be a strong solution to a balance law \eqref{system}, and $V$ any smooth function. Then, the following holds:
\begin{align*}
\begin{aligned} 
&\frac{d}{dt}\int_{\bbt^d}\eta(V|U)\,dx =\frac{d}{dt}\int_{\bbt^d}\eta(V)\,dx-\int_{\bbt^d} \nabla_x\bigl(D\eta(U)\bigr): A(V|U)\,dx \\ 
&\qquad  -\int_{\bbt^d} D\eta(U) \cdot[\partial_tV + {\rm div}_x A(V) - F(V)]\,dx\\
&\qquad  -\int_{\bbt^d} \big[ D^2\eta(U) F(U) (V-U) +D\eta(U)F(V) \big]\,dx.
\end{aligned}
\end{align*}
\end{proposition}

\subsection{Wasserstein distance and representation formulae for solutions of the continuity equation} 
For $p\ge1$, the $p$-Wasserstein distance between two probability measures $\mu_1$
and $\mu_2$ on $\bbr^d$ is defined by
\[
W_p^p(\mu_1,\mu_2):=\inf_{\nu\in\Lambda(\mu_1,\mu_2)}\int_{\bbr^{2d}}|x-y|^{2}\,d\nu(x,y),
\]
where $\Lambda(\mu_1,\mu_2 )$ denotes the set of all probability measures $\nu$ on $\bbr^{2d}$ with marginals $\mu_1$ and $\mu_2$,
i.e, 
\[
\pi_{1\#}\nu = \mu_1,\quad \pi_{2\#}\nu = \mu_2,
\]
where $\pi_1: (x,y)\mapsto x$ and $\pi_2 : (x,y)\mapsto y $ are the natural projections from $\bbr^{d}\times \bbr^{d}$ to $\bbr^{d}$, and $\pi_{\#}\nu$ denotes the push forward of $\nu$ through a map $\pi$, i.e., $\pi_{\#}\nu (B):= \nu(\pi^{-1}(B))$ for any Borel set $B$.
This same definition can be extended to measures on the torus $\bbt^d$ with the understanding that $|x-y|$ denotes the distance on the torus.\\

To make a connection between the $L^2$-distance of velocities and the 2-Wasserstein distance of densities (see Lemma \ref{lem-com}), we will use two different representation formulas for solutions to the continuity equation
\beq\label{continuity}
\partial_t\mu_t + \div_x (u_t\mu_t)=0.
\eeq
Let us recall that, if the velocity field $u_t:\bbr^d\to\bbr^d$ is Lipschitz with respect to  $x$, uniformly in $t$, then for any  $x$ there exists a global-in-time unique characteristic $X$ generated by $u_t$ starting from $x$,\[
\dot X(t,x)=u_t (X (t,x)),\quad X(0,x)=x,
\]
and the solution $\mu_t$ of \eqref{continuity} is the push forward of the initial data $\mu_0$ through $X(t)$, i.e.,
\beq\label{Xode}
\mu_t=X(t)_{\#}\mu_0
\eeq
(e.g., see \cite[Proposition 8.1.8]{A-G-S}).
On the other hand, if the velocity field $u_t$ is not Lipschitz with respect to  $x$, the uniqueness of the characteristics is not guaranteed anymore. Still, a probabilistic representation formula for solutions to \eqref{continuity} holds (recall that a curve of probability measures in $\bbr^d$ is said narrowly continuous if it is continuous in the duality with continuous bounded functions):
\begin{proposition}\label{prop-prob}
For a given $T>0$, let $\mu_t:[0,T]\to \mathcal{P}(\bbr^d)$ be a narrowly continuous solution of \eqref{continuity} for a Borel vector field $u_t$ satisfying 
\[
\int_0^T\int_{\bbr^d} |u_t(x)|^p d\mu_t(x) dt <\infty,\quad \mbox{for some } p>1.
\]
Let $\Gamma_{T}$ denote the space of continuous curves from $[0,T]$ into $\bbr^d$.
Then, there exists a probability measure $\eta$ on $\Gamma_{T}\times \bbr^d$ satisfying the following properties:\\
(i) $\eta$ is concentrated on the set of pairs $(\gamma,x)$ such that $\gamma$ is an absolutely continuous curve solving the ODE 
\[
\dot \gamma(t)=u_t (\gamma (t)),\quad\mbox{for } \mbox{a.e.} ~t\in(0,T),~\mbox{with }\gamma(0)=x;
\]
(ii) $\mu_t$ satisfies 
\[
\int_{\bbr^d} \varphi(x) d\mu_t(x) = \int_{\Gamma_{T}\times\bbr^d} \varphi(\gamma(t)) d\eta(\gamma,x),\quad \forall \,\varphi\in C^0_b (\bbr^d), ~t\in [0,T].
\]
\end{proposition}
Again, this result readily extends on the torus.

Note that, in the case when $u_t$ is Lipschitz, there exists a unique curve $\gamma$ solving the ODE and starting from $x$ (i.e., $\gamma=X(\cdot,x)$), so the measure $\eta$ is given by the formula
$$
d\eta(\gamma,x)=\delta_{\gamma=X(\cdot,x)}\otimes d\mu_0(x).
$$
We refer to \cite[Theorem 8.2.1]{A-G-S} for more details and a proof.

\subsection{Useful inequality} We here present a standard inequality that is used in the proof of Lemma \ref{lem-com}, for the convenience of the reader:
\begin{lemma}
\label{lem:W2 L1}
Let $\rho_1,\rho_2:\bbt^d\to \mathbb R$ be two probability densities.
Then
$$
W_2^2(\rho_1,\rho_2)\leq \frac{d}{8}\|\rho_1-\rho_2\|_{L^1(\bbt^d)}.
$$
\end{lemma}
\begin{proof}
The idea is simple: to estimate the transportation cost from $\rho_1$ to $\rho_2$ it suffices to consider a transport plan that keeps at rest all the mass in common between $\rho_1$ and $\rho_2$ (namely $\min\{\rho_1,\rho_2\}$)
and sends $\rho_1-\min\{\rho_1,\rho_2\}$ onto $\rho_2-\min\{\rho_1,\rho_2\}$ in an arbitrary way.
For instance, assuming without loss of generality that $\rho_1\neq \rho_2$ (otherwise the result is trivial), we set
$$
m:=\|\rho_1-\min\{\rho_1,\rho_2\}\|_{L^1(\bbt^d)}
=\|\rho_2-\min\{\rho_1,\rho_2\}\|_{L^1(\bbt^d)}
=\frac{1}{2}\|\rho_1-\rho_2\|_{L^1(\bbt^d)}>0.
$$
Then, a possible choice of transport plan between $\rho_1$ and $\rho_2$ is given by 
\begin{multline*}
\pi(dx,dy):=\delta_{x=y}(dy)\otimes { \min\{\rho_1(x),\rho_2(x)\}dx}
\\
+ \frac{1}{m}[\rho_1(x)-\min\{\rho_1(x),\rho_2(x)\}]dx \otimes [\rho_2(y)-\min\{\rho_1(y),\rho_2(y)\}]dy.
\end{multline*}
Since the diameter of $\bbt^d$ is bounded by $\sqrt{d}/2,$ we deduce that the $W_2^2$-cost to transport 
 $\rho_1-\min\{\rho_1,\rho_2\}$ onto $\rho_2-\min\{\rho_1,\rho_2\}$
is at most
\begin{multline*}
\int_{\bbt^{2d}} |x-y|^2d\pi(x,y)\\= 
\frac{1}{m}\int_{\bbt^{2d}} |x-y|^2
(\rho_1(x)-\min\{\rho_1(x),\rho_2(x)\})( \rho_2(y)-\min\{\rho_1(y),\rho_2(y)\})\,dx\,dy\\
\leq 
\frac{d}{4}\,\|\rho_1-\min\{\rho_1,\rho_2\}\|_{L^1(\bbt^d)}=
\frac{d}{8}\,\|\rho_1-\rho_2\|_{L^1(\bbt^d)},
\end{multline*}
as desired.
\end{proof}

\section{Structural lemma}
In a general system, we first present some structural hypotheses to provide a Gronwall-type inequality on the relative entropy that is also controlled by 2-Wasserstein distance. \\
$\bullet$ {\bf Hypotheses:}  Let $f^\eps$ be a solution to a given kinetic equation $KE_\eps$ scaled with $\eps>0$ corresponding to a initial data $f^\eps_0$. Let $U^\eps$ and $U^\eps_0$ consist of hydrodynamic variables of $f^\eps$ and $f^\eps_0$ respectively.\\
Let $U$ be a solution to a balance law (as a limit system of $KE_\eps$):
\[
\partial_t U + {\rm div}_x A(U) = F(U),\quad U|_{t=0}=U_0.
\]
\begin{itemize}
\item
(${\mathcal H}1$): The kinetic equation $KE_\eps$ has a kinetic entropy $\mathcal{F}$ such that $\int \mathcal{F}(f^{\eps})(t)\,dx\ge 0$ and
\[
\int \mathcal{F}(f^{\eps})(t)\,dx + \frac{1}{\eps} \int_0^t  D_1(f^{\eps})(s)\,ds + \int_0^t  D_2(f^{\eps})(s)\,ds \le \int_{\bbt^d} \mathcal{F}(f^{\eps}_0)\,dx,
\]
where $D_1, D_2\ge 0$ are some dissipations.
\vspace{0.2cm}
\item
(${\mathcal H}2$): There exists a constant $C>0$ (independent of $\eps$) such that
\[
\int \eta(U^{\eps}_0|U_0)\,dx \le C\eps,\quad \int \big( \mathcal{F}(f^{\eps}_0)- \eta(U^\eps_0)\big)\,dx\le C\eps,\quad \int_{\bbt^d} \mathcal{F}(f^{\eps}_0)\,dx\le C.
\]
\item
(${\mathcal H}3$): The balance law has a convex entropy $\eta$, and the minimization property holds: 
\[
\eta(U^\eps)\le  \mathcal{F}(f^{\eps}).
\]
\item
(${\mathcal H}4$): There exists a constant $C>0$ (independent of $\eps$) such that
\[
\Big|\int \nabla_x\bigl(D\eta(U)\bigr): A(U^\eps|U)\,dx \Big| \le C\int\eta(U^\eps|U)\,dx. 
\]
\item
(${\mathcal H}5$): There exists a constant $C>0$ (independent of $\eps$) such that
\[
\Big|\int D\eta(U) \cdot[\partial_tU^\eps + {\rm div}_x A(U^\eps) - F(U^\eps)]\,dx  \Big| \le C D_1(f^\eps).
\] 
\item
(${\mathcal H}6$): Let $\rho^\eps$ be the hydrodynamic variable of $f^\eps$ as the local mass, and $\rho$ be the corresponding variable for the balance law. Then,
\begin{align*}
\begin{aligned}
&-\int \big[ D^2\eta(U) F(U) (U^\eps-U) +D\eta(U)F(U^\eps) \big]\,dx\\
&\quad \le D_2(f^\eps) + C W_2^2(\rho^\eps,\rho)+ C\int \eta(U^\eps|U) dx.
\end{aligned}
\end{align*}
\item
(${\mathcal H}7$): There exists a constant $C>0$ (independent of $\eps$) such that 
\[
W_2^2(\rho^\eps,\rho)(t) \le C\int_0^t\int \eta(U^\eps|U) dxds +C\eps.
\] 
\end{itemize}

\begin{remark}
1. The hypotheses (${\mathcal H}1$)-(${\mathcal H}5$) provide a basic structure in applying the relative entropy method to hydrodynamic limits as in previous results (for example, \cite{KV1,K-M-T-2,MV}). 
On the other hand, the hypotheses (${\mathcal H}6$)-(${\mathcal H}7$) provide a crucial connection between the relative entropy and Wasserstein distance.\\
2. The (kinetic) entropy inequality (${\mathcal H}1$) plays an important role in controlling the dissipations $D_1, D_2$ in (${\mathcal H}5$) and (${\mathcal H}6$).\\
3. (${\mathcal H}2$) is related to a kind of well-prepared initial data.
\end{remark}

\begin{lemma}\label{lem-gen}
Assume the hypotheses (${\mathcal H}1$)-(${\mathcal H}7$). Then, for a given $T>0$, there exists a constant $C>0$  such that
\[
\int \eta(U^\eps|U)(t) dx + W_2^2(\rho^\eps,\rho) (t) \le C\eps,\quad t\le T.
\]
\end{lemma}
\begin{proof}
First of all, using Proposition \ref{prop-key2}, we have
\[
\int_{\bbt^d}\eta(U^\eps|U)(t)\,dx \le I_1+I_2+I_3+I_4+I_5,
\]
\begin{align*}
\begin{aligned} 
&I_1:=\int_{\bbt^d}\eta(U^\eps_0|U_0)\,dx,\\
&I_2:=\int_{\bbt^d}\big(\eta(U^\eps)(t)-\eta(U^\eps_0)\big)\,dx,\\
&I_3:=-\int_0^t \int_{\bbt^d} \nabla_x\bigl(D\eta(U)\bigr): A(U^\eps|U)\,dxds, \\ 
&I_4:= -\int_0^t \int_{\bbt^d} D\eta(U) \cdot[\partial_tU^\eps + {\rm div}_x A(U^\eps) - F(U^\eps)]\,dxds,\\
&I_5:=  -\int_0^t \int_{\bbt^d} \big[ D^2\eta(U) F(U) (U^\eps-U) +D\eta(U)F(U^\eps) \big]\,dxds.
\end{aligned}
\end{align*}
It follows from (${\mathcal H}2$) that $I_1\le C\eps$.\\
We decompose $I_2$ as
\begin{align}
\begin{aligned} \label{I-2}
 I_2 =  \underbrace{\int_{\bbt^d}(\eta(U^{\eps})(t) - \mathcal{F}(f^{\eps})(t))\,dx}_{=:I_2^1} + \underbrace{\int_{\bbt^d}( \mathcal{F}(f^{\eps})(t)- \mathcal{F}(f^{\eps}_0))\,dx}_{=:I_2^2}
  + \underbrace{\int_{\bbt^d}( \mathcal{F}(f^{\eps}_0)- \eta(U^{\eps}_0))\,dx}_{=:I_2^3}.
\end{aligned}
\end{align}
First, $I_2^1\le0$ by (${\mathcal H}3$).\\
Since (${\mathcal H}1$) yields 
\[
I_2^2\le -\int_0^t D_2(f^\eps) ds,
\]
it follows from (${\mathcal H}6$) that
\[
I_2^2+I_5 \le C\int_0^t W_2^2(\rho^\eps,\rho)ds + C\int_0^t \int_{\bbt^d} \eta(U^\eps|U) dx ds.
\]
By (${\mathcal H}2$), $I_2^3\le C\eps$.\\
It follows from (${\mathcal H}4$) that 
\[
I_3\le C\int_0^t \int_{\bbt^d} \eta(U^\eps|U) dx ds.
\]
Since (${\mathcal H}1$) and (${\mathcal H}2$) imply
\[
 \int_0^t  D_1(f^{\eps})(s)\,ds \le C\eps,
\]
we have $I_4 \le C\eps$.\\
Therefore, we have
\[
\int \eta(U^\eps|U)(t) dx \le C\eps + C\int_0^t \Big[ \int \eta(U^\eps|U)(s) dxds +  W_2^2(\rho^\eps,\rho) \Big] ds.
\]
Hence, combining it with (${\mathcal H}7$), and using Gronwall's inequality, we have the desired result. 
\end{proof}

\section{Proof of Theorem \ref{main thm}}
\label{sect:proof thm} 
The main part of the proof consists in proving the estimate \eqref{main-ineq}. 
\subsection{Proof of \eqref{main-ineq}} 
This will be done by verifying the hypotheses (${\mathcal H}1$)-(${\mathcal H}7$), and then completed by Lemma \ref{lem-gen}.

\subsubsection{\bf Verification of (${\mathcal H}1$):} (${\mathcal H}1$) is satisfied thanks to Lemma \ref{lem-ineq} below. There we show that one can replace the nonlocal dissipation $\mathcal{D}_2$ in the kinetic entropy inequality \eqref{kinetic-ineq} by another dissipation $\tilde {\mathcal{D}}_2$ defined in terms of the hydrodynamic variables $\rho^{\eps}$ and $u^{\eps}$.
\begin{lemma}\label{lem-ineq} 
For any $\eps>0$, assume that $f^{\eps}_0$ satisfies 
\[
f^{\eps}_0 \in L^{1}\cap L^{\infty}(\bbt^d\times\bbr^d),\quad |v|^2 f^{\eps}_0 \in L^{1}(\bbt^d\times\bbr^d).
\]
Then the weak solution $f^{\eps}$ in Proposition \ref{prop-weak} also satisfies
\begin{align}
\begin{aligned} \label{vlasov-e}
&\int_{\bbt^d} \mathcal{F}(f^{\eps})(t)\,dx + \frac{1}{\eps} \int_0^t  \mathcal{D}_1(f^{\eps})(s)\,ds + \int_0^t  \tilde{\mathcal{D}}_2(f^{\eps})(s)\,ds \le \int_{\bbt^d} \mathcal{F}(f^{\eps}_0)\,dx,
\end{aligned}
\end{align}
where $\mathcal{F}$ and $\mathcal{D}_1$ as in \eqref{FDD}, and
\[
\tilde{\mathcal{D}}_2(f^{\eps}) :=  \frac{1}{2}  \int_{\bbt^{2d}} \psi(x-y) \rho^{\eps}(t,x) \rho^{\eps}(t,y) |u^{\eps}(t,x)-u^{\eps}(t,y)|^2 \,dx\,dy.
\] 
\end{lemma}
\begin{proof}
Recalling \eqref{kinetic-ineq}, 
it is enough to show $\tilde{\mathcal{D}}_2(f^{\eps})\le {\mathcal{D}}_2(f^{\eps})$.
We first rewrite $\tilde{\mathcal{D}}_2 (f^{\eps})$ in terms of the mesoscopic variables as follows:
using $\psi(x-y)=\psi(y-x)$, we have
\begin{align*}
\begin{aligned}
\tilde{\mathcal{D}}_2 (f^{\eps}) &= \frac{1}{2}\int_{\bbt^{2d}\times \bbr^{2d}} \psi(x-y) f^{\eps}(t,x,v) f^{\eps}(t,y,w) (v-w)\cdot (u^{\eps}(t,x)-u^{\eps}(t,y)) \,dv\,dw\,dx\,dy\\
&= \int_{\bbt^{2d}\times \bbr^{2d}} \psi(x-y) f^{\eps}(t,x,v) f^{\eps}(t,y,w) (v-w)\cdot u^{\eps}(t,x) \,dv\,dw\,dx\,dy\\
&=\underbrace{\int_{\bbt^{2d}\times \bbr^{2d}} \psi(x-y) f^{\eps}(t,x,v) f^{\eps}(t,y,w) (v-w)\cdot v \,dv\,dw\,dx\,dy}_{=:\mathcal{I}_1}\\
&\quad \underbrace{+ \int_{\bbt^{2d}\times \bbr^{2d}} \psi(x-y) f^{\eps}(t,x,v) f^{\eps}(t,y,w) (v-w)\cdot (u^{\eps}(t,x)-v) \,dv\,dw\,dx\,dy}_{=:\mathcal{I}_2}.
\end{aligned}
\end{align*}
First, we have 
\[
\mathcal{I}_1=\frac{1}{2}\int_{\bbt^{2d}\times \bbr^{2d}} \psi(x-y) f^{\eps}(t,x,v) f^{\eps}(t,y,w) |v-w|^2 \,dx\,dy\,dv\,dw={\mathcal{D}}_2(f^{\eps}).
\]
We next claim $\mathcal{I}_2\le 0$.\\
Indeed, since
\begin{align}
\begin{aligned}\label{pre-mini}
\rho^{\eps} |u^{\eps}|^2 = \frac{\Big(\int_{\bbr^d} v f^{\eps} dv \Big)^2}{\int_{\bbr^d}  f^{\eps} dv} \le \int_{\bbr^d} |v|^2 f^{\eps} dv,
\end{aligned}
\end{align}
we have
\begin{multline*}
\int_{\bbt^{2d}\times \bbr^{2d}} \psi(x-y) f^{\eps}(t,x,v) f^{\eps}(t,y,w) |v|^2 \,dv\,dw\,dx\,dy\\
\ge  \int_{\bbt^{2d}} \psi(x-y) \rho^{\eps}(t,y)\rho^{\eps}(t,x) |u^{\eps}(t,x)|^2 \,dx\,dy.
\end{multline*}
Then, since
\begin{multline*}
\int_{\bbt^{2d}\times \bbr^{2d}} \psi(x-y) f^{\eps}(t,x,v) f^{\eps}(t,y,w) u^{\eps}(t,x)\cdot w  \,dv\,dw\,dx\,dy \\=\int_{\bbt^{2d}} \psi(x-y) \rho^{\eps}(t,x)\rho^{\eps}(t,y) u^{\eps}(t,x)\cdot u^{\eps}(t,y) \,dx\,dy,
\end{multline*}
\begin{multline*}
\int_{\bbt^{2d}\times \bbr^{2d}} \psi(x-y) f^{\eps}(t,x,v) f^{\eps}(t,y,w) u^{\eps}(t,x)\cdot v  \,dv\,dw\,dx\,dy \\= \int_{\bbt^{2d}} \psi(x-y) \rho^{\eps}(t,x)\rho^{\eps}(t,y) |u^{\eps}(t,x)|^2 \,dx\,dy,
\end{multline*}
\begin{multline*}
\int_{\bbt^{2d}\times \bbr^{2d}} \psi(x-y) f^{\eps}(t,x,v) f^{\eps}(t,y,w) v\cdot w \,dv\,dw\,dx\,dy\\=\int_{\bbt^{2d}} \psi(x-y) \rho^{\eps}(t,x)\rho^{\eps}(t,y) u^{\eps}(t,x)\cdot u^{\eps}(t,y) \,dx\,dy,
\end{multline*}
we conclude that $\mathcal{I}_2\le 0$, as desired.
\end{proof}

\subsubsection{\bf Verification of (${\mathcal H}2$):} We show that the assumptions (${\mathcal A}1$)-(${\mathcal A}3$) for initial data imply (${\mathcal H}2$). 
Using \eqref{relative-e} and assumption $({\mathcal A}3)$, we have
\[
\int_{\bbt^d} \eta(U^{\eps}_0|U_0)\,dx = \frac{1}{2}\int_{\bbt^d}\rho^{\eps}_0|u^{\eps}_0 - u_0|^2 dx \le  C\eps^2 \int_{\bbt^d}\rho_0^\eps \,dx \le C\eps^2.
\]
Since it follows from $(\mathcal{A}1)$-$(\mathcal{A}3)$ that 
\[
\int_{\bbt^d}( \mathcal{F}(f^{\eps}_0)- \eta(U_0))\,dx=\mathcal{O}(\eps),
\]
and
\begin{align*}
\begin{aligned}
\int_{\bbt^d}( \eta(U_0)- \eta(U^{\eps}_0))\,dx &= \frac{1}{2} \int_{\bbr^d}\left( \rho_0 |u_0|^2 -\rho_0^\eps |u_0^{\eps}|^2\right) \\
&\leq \frac{1}{2} \int_{\bbt^d} |\rho_0-\rho_0^\eps| |u_0|^2+ \frac{1}{2} \int_{\bbt^d}\rho_0^\eps \left| |u_0^{\eps}|^2 - |u_0|^2\right| =\mathcal O(\eps),
\end{aligned}
\end{align*}
we have
\[
\int_{\bbt^d} \big( \mathcal{F}(f^{\eps}_0)- \eta(U^\eps_0)\big)\,dx=\mathcal O(\eps).
\]
It is obvious that $(\mathcal{A}1)$ implies
\[
\int_{\bbt^d} \mathcal{F}(f^{\eps}_0)\,dx\le C.
\]

\subsubsection{\bf Verification of (${\mathcal H}3$):} It follows from \eqref{pre-mini} that
\beq\label{mini}
\eta(U^\eps)=\rho^\eps \frac{|u^{\eps}|^2}{2} \le \int_{\bbr^d} \frac{|v|^2}{2} f^{\eps} dv = \mathcal{F}(f^{\eps}).
\eeq

\subsubsection{\bf Verification of (${\mathcal H}4$):}
Since
\[
A(U)={ P^T \choose \frac{P\otimes P}{\rho}},
\]
we have
\begin{align*}
\begin{aligned}
DA(U)\cdot(U^{\eps}-U) &= D_{\rho}A(U) (\rho^{\eps}-\rho) + D_{P_i} A(U) (P_i^{\eps} - P_i) \\
&={ ({P^{\eps}} -P)^T \choose -\frac{\rho^{\eps}-\rho}{\rho^2} P\otimes P +\frac{1}{\rho} P\otimes (P^{\eps} -P ) +   \frac{1}{\rho} (P^{\eps} -P )\otimes P },
\end{aligned}
\end{align*}
which yields
\begin{align*}
\begin{aligned}
A(U^{\eps}|U)&={ 0 \choose  \frac{1}{\rho^{\eps}}P^{\eps}\otimes P^{\eps}- \frac{1}{\rho}P\otimes P +\frac{\rho^{\eps}-\rho}{\rho^2} P\otimes P -\frac{1}{\rho} P\otimes (P^{\eps} -P ) -\frac{1}{\rho} (P^{\eps} -P )\otimes P }\\
&= {0 \choose \rho^{\eps} (u^{\eps} -u)\otimes (u^{\eps} -u)}.
\end{aligned}
\end{align*}
Therefore, using \eqref{D-eta} and \eqref{relative-e}, we have
\begin{align*}
\begin{aligned}
\Big|\int \nabla_x\bigl(D\eta(U)\bigr): A(U^\eps|U)\,dx \Big|  &= \Big|\int_0^t\int_{\bbt^d} \rho^{\eps} (u^{\eps} -u)\otimes (u^{\eps} -u) :  \nabla_x u\,dx\,ds \Big| \\
&\le C\|\nabla_x u\|_{L^{\infty}((0,T_*)\times\bbt^d)} \int_0^t\int_{\bbt^d} \eta(U^\eps|U) \,dx\,ds.
\end{aligned}
\end{align*}

\subsubsection{\bf Verification of (${\mathcal H}5$):}
For a weak solution $f^{\eps}$ to \eqref{vlasov}, it follows from \eqref{nonclosed} that $U^{\eps}={\rho^{\eps} \choose P^{\eps}}$ solves the system:
\beq\label{weak system}
\partial_t U^{\eps}  + {\rm div}_x A(U^{\eps}) -F(U^{\eps})={\rm div}_x {0 \choose -\int_{\bbr^d} (v-u^{\eps})\otimes(v-u^{\eps})f^{\eps}dv}.
\eeq
where the equality holds in the sense of distributions (see \eqref{weak-form}). 
Therefore, we have
\begin{align*}
\begin{aligned} 
&\Big|\int D\eta(U) \cdot[\partial_tU^\eps + {\rm div}_x A(U^\eps) - F(U^\eps)]\,dx  \Big| \\
&\quad= \Big|\int_{\bbt^d} \nabla_x u:\Big(\int_{\bbr^d}  (v-u^{\eps})\otimes(v-u^{\eps}) f^{\eps} dv \Big)\,dx \Big| \\
&\quad \le C\|\nabla_x u\|_{L^{\infty}((0,T_*)\times\bbt^d)} \int_{\bbt^d\times\bbr^d} |v-u^{\eps}|^2  f^{\eps} dv\,dx =  C\|\nabla_x u\|_{L^{\infty}((0,T_*)\times\bbt^d)}  \mathcal{D}_1(f^{\eps}).
\end{aligned}
\end{align*}

\subsubsection{\bf Verification of (${\mathcal H}6$):} From the proof of Proposition 4.2 in \cite{K-M-T-2}, we see
\[
-\int_{\bbt^d} \big[ D^2\eta(U) F(U) (U^\eps-U) +D\eta(U)F(U^\eps) \big]\,dx=K_1+K_2+K_3,
\]
where
\begin{align*}
\begin{aligned}
&K_1:= -\frac{1}{2}\int_{\bbt^{2d}} \psi(x-y) \rho^{\eps}(x)\rho^{\eps}(t,y) \big|(u^{\eps}(x) - u(x)) -(u^{\eps}(y)-u(y))\big|^2 \,dx\,dy,\\
&K_2:=\frac{1}{2}\int_{\bbt^{2d}} \psi(x-y) \rho^\eps(x)\rho^\eps(y) |u^\eps(x)-u^\eps(y)|^2 \,dx\,dy,\\
&K_3:=  \int_{\bbt^{2d}} \psi(x-y) \rho^{\eps}(x) ( \rho^{\eps}(y)-\rho(y)) (u(y)-u(x))( u^{\eps}(x)-u(x)) \,dx \,dy.
\end{aligned}
\end{align*}
Notice that $K_1\le 0$, and $K_2=\tilde{\mathcal{D}}_2(f^{\eps})$ where $\tilde{\mathcal{D}}_2(f^{\eps})$ is in Lemma \ref{lem-ineq}.\\
To estimate $K_3$, we separate it into two parts:
\begin{align*}
\begin{aligned}
K_3&=\int_{\bbt^{d}} \Big(\int_{\bbt^{d}}  \psi(x-y) u(y)  (\rho^{\eps}(y)-\rho(y))dy\Big)\rho^{\eps}(x)( u^{\eps}(x)-u(x))\,dx\\
&\quad-\int_{\bbt^{d}} \Big(\int_{\bbt^{d}}  \psi(x-y) (\rho^{\eps}(y)-\rho(y)) dy\Big)u(x)\rho^{\eps}(x)( u^{\eps}(x)-u(x))\,dx.
\end{aligned}
\end{align*}
Since $\psi$ and $u$ are Lipschitz, we use the Kantorovich-Rubinstein Theorem
(see \cite[Theorem 5.10 and Particular Case 5.16]{Villani}) to estimate
\begin{multline*}
K_3\le W_1(\rho^{\eps},\rho)\Big( \sup_{x\in \bbt^d}\|\psi(x-\cdot )u\|_{L^{\infty}(0,T_*;W^{1,\infty}(\bbt^d))}\\
+ \|\psi\|_{L^{\infty}(0,T_*;W^{1,\infty}(\bbt^d))}\|u\|_{L^\infty((0,T_*)\times \bbt^d)} \Big)
 \int_{\bbt^{d}} \rho^{\eps}(x) |u^{\eps}(x)-u(x)|\,  dx.
\end{multline*}
Therefore, since $W_1(\rho^{\eps},\rho) \leq W_2(\rho^{\eps},\rho)$, we obtain
\[
K_3\le C \Big(W_2^2(\rho^{\eps},\rho)+ \int_{\bbt^{d}} \rho^{\eps}(x) |u^{\eps}(x)-u(x)|^2  dx\Big).
\]
Hence we have verified (${\mathcal H}6$).

\subsubsection{\bf Verification of (${\mathcal H}7$):}
This will be shown by Lemma \ref{lem-com} below.
We first derive some estimates for the characteristics generated by the velocity fields $u^{\eps}$ and $u$.\\
For the velocity $u$ in the limit system \eqref{PE}, let $X$ be a characteristic generated by it, that is
\beq\label{X-reg}
\dot X(t,x)=u (t, X (t,x)),\quad X(0,x)=x.
\eeq
Then, thanks to the smoothness of $u$, it follows from \eqref{Xode} that 
\[
X(t)_{\#} \rho_0(x)\,dx =\rho(t,x)\,dx.
\]

On the other hand, since $u^{\eps}$ is not Lipschitz w.r.t $x$, we use a probabilistic representation for $\rho^{\eps}$ as a solution of the continuity equation in \eqref{system}. More precisely, \eqref{mini} and \eqref{kinetic-ineq} imply
\[
\int_{\bbt^d} |u^{\eps}(t)|^2 \rho^{\eps}(t)\, dx \le \int_{\bbt^d}\mathcal{F}(f^{\eps})(t)\,dx\le \int_{\bbt^d}\mathcal{F}(f^{\eps}_0)\,dx<\infty,
\]
so it follows from Proposition \ref{prop-prob} that there exists a probability measure $\eta^{\eps}$ in $\Gamma_{T_*}\times \bbt^d$ that is concentrated on the set of pairs $(\gamma,x)$ such that $\gamma$ is a solution of the ODE 
\beq \label{eq:ODE eps}
\dot \gamma(t)=u^{\eps} (\gamma (t)),\quad \gamma(0)=x,
\eeq
and  
\beq\label{result2}
\int_{\bbt^d} \varphi(x) \rho^{\eps}(t,x)\,dx = \int_{\Gamma_{T_*}\times\bbt^d} \varphi(\gamma(t))\, d\eta^{\eps}(\gamma,x),\quad \forall \,\varphi\in C^0 (\bbt^d), ~t\in [0,T_*].
\eeq
In particular, this says that the time marginal of the measure $\eta^{\eps}$ at time $0$
is given by $\rho^\eps(0)=\rho^{\eps}_0$.
Hence, by the disintegration theorem of measures (see for instance \cite[Theorem 5.3.1]{A-G-S} and the comments at the end of Section 8.2 in \cite{A-G-S}), we can write
$$
d\eta^\eps(\gamma,x)=\eta_x^\eps(d\gamma)\otimes  \rho_0^\eps(x)\,dx,
$$
where $\{\eta_x^\eps\}_{x\in \bbt^d}$ is a family of probability measures on $\Gamma_{T^*}$ concentrated on solutions of \eqref{eq:ODE eps}.\\

For the flow $X$ in \eqref{X-reg}, we also consider the densities $\widetilde{\rho}^\eps(t)$ defined as
\begin{equation}
\label{eq:tilde rho eps}
\widetilde{\rho}^\eps(t,x)\,dx=X(t)_\#\rho_0^\eps(x)\,dx.
\end{equation}
Note that, since
\begin{multline*}
\|\rho(t)-\widetilde{\rho}^\eps(t)\|_{L^1(\bbt^d)}
=\sup_{\|\varphi\|_{\infty}\leq 1} \int_{\bbt^d} \varphi(x)[\rho(t,x)-\widetilde{\rho}^\eps(t,x)]\,dx\\
=\sup_{\|\varphi\|_{\infty}\leq 1} \int_{\bbt^d} \varphi(X(t,x))[\rho_0(x)-{\rho}_0^\eps(x)]\,dx
\leq \|\rho_0^\eps-\rho_0\|_{L^1(\bbt^d)}.
\end{multline*}
we have
\begin{equation}
\label{eq:L1 eps}
\|\rho(t)-\widetilde{\rho}^\eps(t)\|_{L^1(\bbt^d)}\leq \|\rho_0^\eps-\rho_0\|_{L^1(\bbt^d)}.
\end{equation}

We now consider the measure $\nu^\eps$ 
on $ \Gamma_{T_*}\times\Gamma_{T_*}\times \bbt^d$ defined as
$$
d\nu^\eps(\gamma,\sigma,x)=
\eta_x^\eps(d\gamma)\otimes \delta_{X(\cdot,x)}(d\sigma)\otimes \rho_0^\eps(x)\,dx.
$$
If we consider the evaluation map
$$
E_t:\Gamma_{T_*}\times\Gamma_{T_*}\times \bbt^d\to\bbt^d\times \bbt^d,\qquad
E_t(\gamma,\sigma,x)=(\gamma(t),\sigma(t)),
$$
it follows that the measure $\pi^\eps_t:=(E_t)_\# \nu^\eps$ on $\bbt^d\times \bbt^d$ has marginals  $\rho^\eps(t,x)\,dx$
and $\widetilde{\rho}^\eps(t,y)\,dy$ for all $t \geq 0$. Therefore, we have
\begin{align}
\begin{aligned}\label{XW}
\int_{\Gamma_{T_*}\times \bbt^d} |\gamma(t)-X(t,x)|^2 \eta^\eps_x(d\gamma) \otimes \rho_0^{\eps}(x)\,dx
&=\int_{\Gamma_{T_*}\times\Gamma_{T_*}\times \bbt^d}|\gamma(t)-\sigma(t)|^2d\nu^\eps(\gamma,\sigma,x)\\
& = \int_{\bbt^{2d}} |x-y|^2 d\pi_t^\eps(x,y)\\
& \ge W^2_2 (\rho^\eps(t),\widetilde{\rho}^{\eps}(t)).
\end{aligned}
\end{align}

We now use the above results to prove the following lemma.
\begin{lemma}\label{lem-com} 
Under the same assumptions as in Theorem \ref{main thm}, we have that 
\beq\label{Wa-est}
W^2_2 ({\rho}^{\eps}(t),\rho(t)) \le Ce^{T_*}\int_0^t \int_{\bbt^d} |u^{\eps}(s,x) -u(s,x) |^2 \rho^{\eps}(s,x)\,dx\,ds
+ \mathcal O(\eps),\qquad t\le T_*.
\eeq
\end{lemma}
\begin{proof}
Let $\widetilde{\rho}^\eps$ be defined as \eqref{eq:tilde rho eps}.
We begin by observing that, thanks to Lemma \ref{lem:W2 L1}, \eqref{eq:L1 eps}, and assumption (${\mathcal A}2$), it follows that
$$
W_2^2(\widetilde{\rho}^\eps(t),\rho(t)) \leq \mathcal O(\eps).
$$
Hence, to prove \eqref{Wa-est}, it is enough to bound
$W^2_2 ({\rho}^{\eps}(t),\widetilde{\rho}^\eps(t))$.

To this aim, we try to get a Gronwall-type inequality on 
\[\int_{\Gamma_{T_*}\times \bbt^d} |\gamma(t)-X(t,x)|^2 \eta^\eps_x(d\gamma) \otimes \rho_0^{\eps}(x)\,dx.\] 
Since 
\[
\dot\gamma(t)-\dot X(t,x) = \big(u^{\eps} (\gamma(t))-u(\gamma(t))\big)+\big(u (\gamma(t))-u (X(t,x))\big),
\]
(by \eqref{X-reg} and \eqref{eq:ODE eps}),
we have
\begin{align*}
\begin{aligned} 
&\frac{1}{2}\frac{d}{dt}\int_{\Gamma_{T_*}\times \bbt^d} |\gamma(t)-X(t,x)|^2 d\eta^\eps_x(\gamma) \otimes \rho_0^{\eps}(x)\,dx \\
 &\le \int_{\Gamma_{T_*}\times \bbt^d} |u^{\eps} (\gamma(t))-u (\gamma(t))|^2 d\eta_x^\eps(\gamma) \otimes \rho_0^{\eps}(x)\,dx \\
&\quad + \int_{\Gamma_{T_*}\times \bbt^d}|u (\gamma(t))-u (X(t,x))|^2 d\eta_x^\eps(\gamma) \otimes \rho_0^{\eps}(x)\,dx \\
 &\quad+ 2\int_{\Gamma_{T_*}\times \bbt^d} |\gamma(t)-X(t,x)|^2 d\eta_x^\eps(\gamma) \otimes \rho_0^{\eps}(x)\,dx.
\end{aligned}
\end{align*}
Notice that, thanks to \eqref{result2},
\[
\int_{\Gamma_{T_*}\times \bbt^d} |u^{\eps} (\gamma(t))-u (\gamma(t))|^2 d\eta_x^\eps(\gamma) \otimes \rho_0^{\eps}(x)\,dx  =\int_{\bbt^d} |u^{\eps}(t,x) -u(t,x) |^2 \rho^{\eps}(t,x)\,dx.
\]
Moreover, since
\begin{multline*}
\int_{\Gamma_{T_*}\times \bbt^d}|u (\gamma(t))-u (X(t,x))|^2 d\eta_x^\eps(\gamma) \otimes \rho_0^{\eps}(x)\,dx\\  \le  \|u\|_{L^{\infty}(0,T_*;W^{1,\infty}(\bbt^d))}\int_{\Gamma_{T_*}\times \bbt^d} |\gamma(t)-X(t,x)|^2 d\eta_x^\eps(\gamma) \otimes \rho_0^{\eps}(x)\,dx,
\end{multline*}
we have
\begin{multline*}
\frac{d}{dt}\int_{\Gamma_{T_*}\times \bbt^d} |\gamma(t)-X(t,x)|^2 d\eta_x^\eps(\gamma) \otimes \rho_0^{\eps}(x)\,dx\\
 \le C \int_{\Gamma_{T_*}\times \bbt^d} |\gamma(t)-X(t,x)|^2 d\eta_x^\eps(\gamma) \otimes \rho_0^{\eps}(x)\,dx+\int_{\bbt^d} |u^{\eps}(t,x) -u(t,x) |^2 \rho^{\eps}(t,x)\,dx.
\end{multline*}
Therefore, using Gronwall's inequality together with $\gamma(0)=X(0,x)=x$ for $\eta_x^\eps$-a.e. $\gamma$, we obtain
\begin{multline*}
\int_{\Gamma_{T_*}\times \bbt^d} |\gamma(t)-X(t,x)|^2 d\eta_x^\eps(\gamma) \otimes \rho_0^{\eps}(x)\,dx\\ \le Ce^{T_*}\int_0^t \int_{\bbt^d} |u^{\eps}(s,x) -u(s,x) |^2 \rho^{\eps}(s,x)\,dx\,ds,\qquad t\le T_*.
\end{multline*}
Hence, using \eqref{XW} we get the desired control on $W^2_2 ({\rho}^{\eps}(t),\widetilde{\rho}^\eps(t))$, which concludes the proof.
\end{proof}

\subsection{Proof of \eqref{w-conv}}
Here we use the estimate \eqref{main-ineq} to show the convergence \eqref{w-conv}. \\
First, since \eqref{vlasov-e} and (${\mathcal A}1$) imply
\[
 \int_0^t  \mathcal{D}_1(f^{\eps})(s)\,ds \le C\eps,
\]
using \eqref{main-ineq}, we have
\begin{align}
\begin{aligned}\label{inst0}
\int_0^{T_*} \int_{\bbt^{d}\times\bbr^d} f^{\eps}|v-u|^2\,dx\,dv\,ds &\le 2\int_0^{T_*} \int_{\bbt^{d}\times\bbr^d} f^{\eps}(|v-u^{\eps}|^2 + |u^{\eps}-u|^2)\,dx\,dv\,ds\\
& \le C(1+T_*)\eps.
\end{aligned}
\end{align}
Then, for any $\phi\in C^1_c((0,T_*)\times\bbt^{d}\times\bbr^d)$, 
\begin{align*}
\begin{aligned} 
&\int_0^{T_*} \int_{\bbt^{d}\times\bbr^d} \phi(s,x,v) f^{\eps}\,dx\,dv\,ds -\int_0^{T_*} \int_{\bbt^{d}\times\bbr^d} \phi(s,x,v)\rho\,\delta_{u}(dv) \,dx\,ds \\
&\quad = \int_0^{T_*} \int_{\bbt^{d}\times\bbr^d} \phi(s,x,v) f^{\eps}\,dx\,dv\,ds -\int_0^{T_*} \int_{\bbt^{d}} \phi(s,x,u)\rho \,dx\,ds\\
&\quad = \underbrace{\int_0^{T_*}\int_{\bbt^{d}\times\bbr^d}f^{\eps} \big(\phi(s,x,v)-\phi(s,x,u)\big)\,dx\,dv\,ds}_{=:I_1^{\eps}} +\underbrace{\int_0^{T_*} \int_{\bbt^{d}} \phi(s,x,u)(\rho^{\eps}-\rho) \,dx\,ds}_{=:I_2^{\eps}}.
\end{aligned}
\end{align*} 
Using \eqref{inst0}, we have
\begin{align*}
\begin{aligned} 
I_1^{\eps}&\le \|\nabla_v\phi\|_{\infty} \int_0^{T_*} \int_{\bbt^{d}\times\bbr^d} f^{\eps} |v-u|\,dx\,dv\,ds\\
&= \|\nabla_v\phi\|_{\infty} \Big(\int_0^{T_*} \int_{|v-u|\le \sqrt{\eps}} f^{\eps} |v-u|\,dx\,dv\,ds
+\int_0^{T_*} \int_{|v-u|> \sqrt{\eps}}f^{\eps} |v-u|\,dx\,dv\,ds \Big)\\
&\le \|\nabla_v\phi\|_{\infty}  \Big(\sqrt{\eps} T_*
+ \frac{1}{\sqrt{\eps}}\int_0^{T_*} \int_{|v-u|> \sqrt{\eps}} f^{\eps} |v-u|^2 dv\,dx\,ds \Big)\\
&\le C(1+T_*) \sqrt{\eps}.
\end{aligned}
\end{align*} 
Since $W_1 (\rho^{\eps},\rho)\leq W_2 (\rho^{\eps},\rho)\to0$ by \eqref{main-ineq}, we also have $I_2^{\eps}\to0$ as $\eps\to0$.\\
Hence, this completes the proof of  \eqref{w-conv}.

\appendix

\section{Proof of Theorem \ref{prop-flocking}}
We first estimate $\frac{d}{dt}{\mathcal E}_1$ as follows:
\begin{align*}
\begin{aligned} 
\frac{d}{dt}{\mathcal E}_1&= 2\int_{\bbt^d\times\bbr^d}f (u-v)\partial_tu  \,dv\,dx+ \int_{\bbt^d\times\bbr^d}\partial_t f |u-v|^2 \,dv\,dx \\
&:=I_1+I_2.
\end{aligned}
\end{align*}
First of all, by the definition of $u$, we have $\int f(u-v)\,dv=0$, hence $I_1=0$.\\
Concerning $I_2$, it follows from \eqref{main} that
\begin{align*}
\begin{aligned} 
I_2&=  \int_{\bbt^d\times\bbr^d} |u-v|^2 \Big(-\nabla_x\cdot(v f) - \nabla_v\cdot( L[f] f) - \nabla_v\cdot ((u-v)f) \Big)\,dv\,dx \\
&=  \underbrace{2\int_{\bbt^d\times\bbr^d}   \nabla_x u (u-v) \cdot v f \,dv\,dx}_{=:I_{21}}  \underbrace{-2\int_{\bbt^d\times\bbr^d}(u-v)\cdot L[f] f \,dv\,dx}_{=:I_{22}}   \\
&\quad  \underbrace{- 2\int_{\bbt^d\times\bbr^d} |u-v|^2  f \,dv\,dx}_{=-2\mathcal E_1}.
\end{aligned}
\end{align*}
Then, we use the stress tensor $P= \int_{\bbr^d} (v-u)\otimes (v-u) f\, dv$ to rewrite $I_{21}$ as
\begin{align*}
\begin{aligned} 
I_{21} =2\int_{\bbt^d\times\bbr^d}   \nabla_x u  (u-v) \cdot (v-u) f \,dv\,dx=2\int_{\bbt^d}  (\nabla_x\cdot P) \cdot u \,dx.
\end{aligned}
\end{align*}
Thanks to the estimate on $\mathcal{I}_2$ in the proof of Lemma \ref{lem-ineq}, we see that 
\[
I_{22}=-2\int_{\bbt^{2d}\times \bbr^{2d}} \psi(x-y) f(t,x,v) f(t,y,w) (u(t,x)-v)\cdot(w-v) \,dv\,dw\,dx\,dy\le 0. 
\]
Therefore, we have
\begin{equation}
\label{eq:der E1}
\frac{d}{dt}{\mathcal E}_1\le 2\int_{\bbt^d}  (\nabla_x\cdot P) \cdot u\,dx- 2{\mathcal E}_1.
\end{equation} 
We next estimate $\frac{d}{dt}{\mathcal E}_2$ as follows:
\begin{align*}
\begin{aligned} 
\frac{d}{dt}{\mathcal E}_2&= 2\int_{\bbt^{2d}}\partial_t\rho(t,x)\rho(t,y) |u(t,x) - u(t,y)|^2 \,dx\,dy \\
&\quad + 2\int_{\bbt^{2d}}\rho(t,x)\rho(t,y) (u(t,x) - u(t,y)) \partial_t (u(t,x) - u(t,y)) \,dx\,dy \\
&:=J_1 + J_2.
\end{aligned}
\end{align*}
Since it follows from \eqref{nonclosed} with $\eps=1$ that
\begin{align*}
\begin{aligned}
&\partial_t\rho + \nabla_x \cdot (\rho u)=0,\\
&\rho\partial_tu +\rho u\cdot\nabla_x u +\nabla_x\cdot P  =\int_{\bbr^d}L[f] fdv,\\
\end{aligned}
\end{align*}
we obtain (recall that $\|\rho\|_{L^1(\bbt^d)}=1$)
\begin{align*}
\begin{aligned} 
J_1&=-2\int_{\bbt^{2d}}\nabla_x \cdot (\rho u)(t,x)\rho(t,y) |u(t,x) - u(t,y)|^2 \,dx\,dy\\
&=4\int_{\bbt^{d}}\rho u\cdot \nabla_x u \cdot u \,dx - 4\int_{\bbt^{d}}\rho u\cdot \nabla_x u \,dx \cdot \int_{\bbt^{d}} \rho u\, dx,
\end{aligned}
\end{align*}
and 
\begin{align*}
\begin{aligned} 
J_2&= 4\int_{\bbt^{2d}}\rho(t,y) u(t,x) \rho(t,x) \partial_t u(t,x) \,dx\,dy -4\int_{\bbt^{2d}}\rho(t,y) u(t,y) \rho(t,x) \partial_t u(t,x) \,dx\,dy\\
&=-4\int_{\bbt^{d}}\rho u\cdot \nabla_x u \cdot u\,dx -4\int_{\bbt^{d}}\nabla_x\cdot P \cdot u\,dx\\
&\quad +4\underbrace{\int_{\bbt^{d}\times\bbr^d} u\cdot L[f] f\,dx\,dv}_{:=J_{21}} +4\int_{\bbt^{d}}\rho u\cdot \nabla_x u\,dx \cdot \int_{\bbt^{d}} \rho u\,dx\\
&\quad + 4\underbrace{\int_{\bbt^{d}}\nabla_x \cdot P dx}_{=0} \cdot \int_{\bbt^{d}} \rho u\,dx -4\underbrace{\int_{\bbt^{d}\times\bbr^d} L[f] f\,dx\,dv}_{:=J_{22}} \cdot \int_{\bbt^{d}} \rho u\,dx. 
\end{aligned}
\end{align*}
Now, we compute the above terms $J_{21}$ and $J_{22}$ as follows:
\begin{align*}
\begin{aligned} 
J_{21}&= \int_{\bbt^{2d}\times \bbr^{2d}} \psi(x-y) f(t,x,v) f(t,y,w) (w-v)\cdot u(t,x) \,dv\,dw\,dx\,dy \\
&= \int_{\bbt^{2d}} \psi(x-y) \rho(t,x)\rho(t,y) (u(t,y)-u(t,x))\cdot u(t,x) \,dx\,dy\\
&=-\frac{1}{2}\int_{\bbt^{2d}} \psi(x-y) \rho(t,x)\rho(t,y) |u(t,x)-u(t,y)|^2 \,dx\,dy,\\
J_{22}&=\int_{\bbt^{2d}\times \bbr^{2d}} \psi(x-y) f(t,x,v) f(t,y,w) (w-v) \,dv\,dw\,dx\,dy =0.
\end{aligned}
\end{align*}
Therefore, we have
$$
\frac{d}{dt}{\mathcal E}_2= -4\int_{\bbt^{d}}\nabla_x\cdot P \cdot u\,dx
-2\int_{\bbt^{2d}} \psi(x-y) \rho(t,x)\rho(t,y) |u(t,x)-u(t,y)|^2 \,dx\,dy.
$$
Recalling \eqref{eq:der E1}, proves that 
\begin{align*}
\begin{aligned} 
\frac{d}{dt}{\mathcal E}&\le - 2{\mathcal E}_1-\int_{\bbt^{2d}} \psi(x-y) \rho(t,x)\rho(t,y) |u(t,x)-u(t,y)|^2 \,dx\,dy\\
&\le  -2{\mathcal E}_1 -\psi_m{\mathcal E}_2 \le -2\min\{1,\psi_m \} {\mathcal E},
\end{aligned}
\end{align*}
which completes the proof of \eqref{eq:E}.

To show the second bound \eqref{eq:E1}, note that if $\ell_T:=\sup_{t\in[0,T]}\|\nabla_x u\|_{L^\infty(\mathbb T^d)}<\infty$ then
\eqref{eq:der E1} yields
$$
\frac{d}{dt}\mathcal E_1(t)\leq -2\int_{\mathbb T^d}\nabla_xu:P\,dx-2\mathcal E_1\leq 2\ell_T\int_{\mathbb T^d\times\bbr^d}|u-v|^2f\,dv\,dx-2\mathcal E_1(t)=2(\ell_T-1)\mathcal E_1(t),
$$
which proves \eqref{eq:E1}.
\qed

\bibliographystyle{amsplain}
\bibliography{hydrodynamic_limit_CS_APDE_revision2}

\end{document}